\documentclass[11pt]{amsart}
\usepackage{amssymb,amscd,verbatim,
mathrsfs, latexsym,amsmath,amsthm}
\usepackage{a4wide}
\usepackage{color}

\numberwithin{equation}{section}
\newtheorem{theorem}{Theorem}[section]
\newtheorem{proposition}[theorem]{Proposition}
\newtheorem{lemma}[theorem]{Lemma}
\newtheorem{corollary}[theorem]{Corollary}
\theoremstyle{definition}
\newtheorem{definition}[theorem]{Definition}

\newtheorem{prob}[theorem]{Problem}

\def\Lra{\Leftrightarrow}

\def\tr{\text{tr}}
\DeclareMathOperator{\IM}{Im} \DeclareMathOperator{\RE}{Re}

\def\sR{\hbox{I\kern-.1667em\hbox{R}}}

\newcommand{\R}{\mathbb R}

\newcommand{\C}{\mathbb C}

\newcommand{\Z}{\mathbb Z}

\newcommand{\N}{\mathbb N}
\newcommand{\Q}{\mathbb Q}
\newcommand{\bH}{\mathbb H}

\def\tr{\hbox{Tr}}

\def\tr{\mathrm{tr}}

\begin{document}
\title[Determinants of Laplacians on Hilbert modular surfaces]
{Determinants of Laplacians on Hilbert modular surfaces}

\author[Y. Gon]{Yasuro Gon}
\email{ygon@math.kyushu-u.ac.jp}
\address{Faculty of Mathematics\\ Kyushu University\\
744 Motooka, Nishi-ku \\ Fukuoka 819-0395\\ Japan}

\thanks{2010 Mathematics Subject Classification. 11M36, 11F72, 58J52 \\
This work was partially supported by JSPS Grant-in-Aid for Scientific Research (C) no. 26400017.}

\date{\today}

\begin{abstract}
{We study regularized determinants of Laplacians acting on the space of Hilbert-Maass forms 
for the Hilbert modular group of a real quadratic field. We show that these determinants are described by
Selberg type zeta functions introduced in \cite{G0,G1}. }
\end{abstract}

\keywords{Hilbert modular surface; Selberg zeta function; Regularized determinant.}

\maketitle

\setcounter{tocdepth}{1}

\section{Introduction}

Determinants of the Laplacian $\Delta$ acting on the space of Maass forms on a hyperbolic Riemann surface $X$
are studied by many authors. (See, for example \cite{Sa1, E1, Ku, Ko}.) It is known that
the determinants of $\Delta$ are described by the Selberg zeta function (cf. \cite{Sel1}) for $X$.

On the other hand, two Laplacians $\Delta^{(1)}$, $\Delta^{(2)}$ act on the space of Hilbert-Maass forms 
on the Hilbert modular surface $X_{K}$ of  a real quadratic field $K$.
For this reason, it seems that there are no explicit formulas for ``Determinants of Laplacians'' on $X_K$ until now.
In this article we consider regularized determinants of the first Laplacian $\Delta^{(1)}$ acting on 
its certain subspaces $V_m^{(2)}$, indexed by $m \in 2 \N$. We show that
these determinants are described by
Selberg type zeta functions for $X_K$ introduced in \cite{G0,G1}. 

Let
$K/\Q$ be a real quadratic field with class number one and
$\mathcal{O}_K$ be the ring of integers of $K$. Put $D$ be the
discriminant of $K$ and $\varepsilon > 1 $ be the fundamental
unit of $K$. 
We denote the generator of $\mathrm{Gal}(K/\Q)$
by $\sigma$ and put $a' := \sigma(a)$ 
for $a \in K$. 
We also put
$\gamma' =  \Bigl( 
\begin{array}{cc}
a' & b' \\
c' & d'
\end{array} \Bigr)$ for 
$\gamma =  \Bigl( 
\begin{array}{cc}
a & b \\
c & d
\end{array} \Bigr) \in \mathrm{PSL}(2, \mathcal{O}_K)$.
Let $\Gamma_{K} = \{ (\gamma, \gamma') \, | \, 
\gamma \in \mathrm{PSL}(2, \mathcal{O}_K) \}$ 
be the Hilbert modular group of $K$.
It is known that $\Gamma_{K}$ is a
co-finite ({\it non-cocompact}) irreducible discrete subgroup 
of $\mathrm{PSL}(2, \R) \times \mathrm{PSL}(2, \R)$ and $\Gamma_K$ acts on 
the product $\bH^2$ of two copies of the upper half plane $\bH$ 
by component-wise linear fractional transformation.
$\Gamma_K$ have only one cusp $(\infty,\infty)$, i.e. 
$\Gamma_K$-inequivalent parabolic fixed point.
$X_K := \Gamma_K \backslash \bH^2$ 
is called the Hilbert modular surface.

Let $(\gamma, \gamma') \in \Gamma_{K}$ be hyperbolic-elliptic,
i.e, $|\tr(\gamma)|>2$ and $|\tr(\gamma')|<2$. 
Then the centralizer of hyperbolic-elliptic 
$(\gamma, \gamma')$ in $\Gamma_K$ is infinite cyclic.

\begin{definition}[Selberg type zeta function for $\Gamma_{K}$ with the weight $(0,m)$]
For an even integer $m \ge 2$, we define
\label{def:zeta}
\begin{equation}
Z_{m}(s) := \prod_{(p,p')\in \text{P}\Gamma_{\text{HE}}} 
\prod_{n=0}^{\infty} \Bigl( 1-e^{i (m-2) \omega} \, N(p)^{-(n+s)} 
\Bigr)^{-1} 
\quad \mbox{for $\RE(s) >1$.} 
\end{equation}
\end{definition}

Here, $(p,p')$ run through the set of primitive hyperbolic-elliptic 
$\Gamma_K$-conjugacy classes of $\Gamma_K$, and $(p,p')$ is conjugate 
in $\mathrm{PSL}(2,\R)^2$ to
\[ (p, p') \sim \Bigl( 
\Bigl( 
\begin{array}{cc}
N(p)^{1/2} & 0 \\
0 & N(p)^{-1/2}
\end{array} \Bigr), 
\, 
\Bigl( 
\begin{array}{cc}
\cos \omega & - \sin \omega \\
\sin \omega & \cos \omega
\end{array} \Bigr)\Bigr).
\]
Here, $N(p)>1$, $\omega \in (0,\pi)$ and $\omega \notin \pi \Q$.
The product is absolutely convergent for $\RE(s)>1$.

Analytic properties of $Z_m(s)$ are known.
\begin{theorem}[{\cite[Theorems 5.3 and 6.5]{G1}}] \label{th:s1} 
For an even integer $m \ge 2$, 
$Z_m(s)$ a priori defined for 
$\RE(s) >1$ has a meromorphic extension over the whole complex plane. 
\end{theorem}

In this article, we also consider ``the square root of $Z_2(s)$''. 
\begin{definition}[$\sqrt{Z_{2}(s)}$]
\begin{equation}
\begin{split}
\sqrt{Z_{2}(s)} &:= \prod_{(p,p') \in \text{P}\Gamma_{\text{HE}}} 
\prod_{n=0}^{\infty} \Bigl( 1-N(p)^{-(n+s)} 
\Bigr)^{-1/2}  \\
&=   \exp  \Bigl(    \frac{1}{2} \sum_{(p,p')}  
\sum_{k=1}^{\infty}  \frac{1}{k} \frac{N(p)^{-ks}}{1-N(p)^{-k}} 
\Bigr)   \quad \mbox{for $\RE(s) >1$}.
\end{split} 
\end{equation}
\end{definition}

By \cite[Theorem 6.5]{G1} and the fact that the Euler characteristic of $X_K$ is even
(See Lemma \ref{EC}), we see that $\frac{d}{ds} \log Z_2(s)$ has even integral residues at any poles.
Therefore, we find that $\sqrt{Z_2(s)}$ has a meromorphic continuation to the whole complex plane.

Let us introduce the completed Selberg type zeta functions $\widehat{Z}_2^{\frac12}(s)$ and $\widehat{Z}_m(s)$ $(m \ge 4)$, 
which are invariant under $s \to 1-s$. (See \cite[Theorems 5.4 and 6.6]{G1}.)
\begin{definition}[Completed Selberg zeta functions] \label{zhat}
\begin{equation} \label{z2hat}
\widehat{Z}_2^{\frac12}(s) := \sqrt{Z_2(s)}  \, Z_{\mathrm{id}}^{\frac12} (s)  \, Z_{\mathrm{ell}}^{\frac12} (s;2) 
                                       \, Z_{\mathrm{par/sct}}^{\frac12} (s;2) \, Z_{\mathrm{hyp2/sct}}^{\frac12} (s;2) 
\end{equation}
with
\[ Z_{\mathrm{id}}^{\frac12} (s) := \bigl( \Gamma_2(s) \Gamma_2(s+1) \bigr)^{\zeta_K(-1)} ,
\quad Z_{\mathrm{ell}}^{\frac12} (s;2) :=  \prod_{j=1}^{N} \prod_{l=0}^{\nu_j-1} \Gamma \bigl( \tfrac{s+l}{\nu_j} \bigr)^{\frac{\nu_j-1-2l}{2 \nu_j}},  
\]
\[  Z_{\mathrm{par/sct}}^{\frac12} (s;2) :=  \varepsilon^{-s}, \quad  Z_{\mathrm{hyp2/sct}}^{\frac12} (s;2) := \zeta_{\varepsilon}(s). 
\]
\begin{equation} \label{zmhat}
\widehat{Z}_m(s) := Z_m(s)  \, Z_{\mathrm{id}}(s)  \, Z_{\mathrm{ell}} (s;m) 
                                       \, Z_{\mathrm{hyp2/sct}}(s;m)  \quad (m \ge 4)
\end{equation}
with
\[ Z_{\mathrm{id}}(s) := \bigl( \Gamma_2(s) \Gamma_2(s+1) \bigr)^{2 \zeta_K(-1)} ,
\quad Z_{\mathrm{ell}}(s;m) :=  \prod_{j=1}^{N} \prod_{l=0}^{\nu_j-1} \Gamma \bigl( \tfrac{s+l}{\nu_j} \bigr)^{\frac{\nu_j-1- \alpha_l(m,j)- \overline{\alpha_l}(m,j)}{\nu_j}},  
\]
\[  Z_{\mathrm{hyp2/sct}} (s;m) := \zeta_{\varepsilon}\Bigl(s+\frac{m}{2}-1\Bigr)  \zeta_{\varepsilon}\Bigl(s+\frac{m}{2}-2\Bigr)^{-1} . 
\]
Here, $\Gamma_2(s)$ is the double Gamma function (for definition, we refer to \cite{KK} 
or \cite[Definition 4.10, p. 751]{GP}), the natural numbers $\nu_1,\nu_2,\dots,\nu_N$ are the orders of the elliptic fixed points in $X_K$
and the integers $\alpha_l(m,j), \overline{\alpha_l}(m,j) \in \{ 0,1,\dots,\nu_j-1\}$ are defined in (\ref{eq:alpha}),
$\zeta_{K}(s)$ is the Dedekind zeta function of $K$, 
$\zeta_{\varepsilon}(s):=(1-\varepsilon^{-2s})^{-1}$ and $\varepsilon$ is the fundamental unit of $K$.
\end{definition}

Let $m \in 2 \N$.  
We recall that two Laplacians 
\begin{equation}
 \Delta_{0}^{(1)} := -y_1^2 \Bigl( \frac{\partial^2}{\partial x_1^2}
   +\frac{\partial^2}{\partial y_1^2} \Bigr), 
   \quad 
   \Delta_{m}^{(2)} := -y_2^2 \Bigl( \frac{\partial^2}{\partial x_2^2}
   +\frac{\partial^2}{\partial y_2^2} \Bigr) 
+ i  m \, y_2 \frac{\partial}{\partial x_2}
\end{equation}
are acting on $L^2_\text{dis} ( \Gamma_{K} \backslash \mathbb{H}^2 ; (0,m) )$, 
the space of Hilbert-Maass forms for $\Gamma_K$ with weight $(0,m)$. (See Definition \ref{def:HM}.)
We consider a certain subspace of $L^2_\text{dis} ( \Gamma_{K} \backslash \mathbb{H}^2 ; (0,m) )$
given by 
\begin{equation}
V_{m}^{(2)} = 
\Bigl\{  f (z_1,z_2) \in  
L^2_\text{dis} ( \Gamma_{K} \backslash \mathbb{H}^2 
; (0,m) )
\Bigl| \,  \Delta_{m}^{(2)} f  = 
\frac{m}{2} \bigl( 1-\frac{m}{2} \bigr) \, f  \Bigr\}.
\end{equation}

The set of eigenvalues of $\Delta_0^{(1)} \big|_{V_m^{(2)}}$ are enumerated as
\[  0 < \lambda_0(m) \le \lambda_1(m) \le \cdots \le \lambda_n(m) \le \cdots \]
Let $s$ be a fixed sufficiently large real number.
We consider the spectral zeta function by using these eigenvalues.
\begin{equation} 
\zeta_m(w, s)=\sum_{n=0}^{\infty} \frac{1}{\bigl( \lambda_n(m)+s(s-1) \bigr)^w}
\quad (\RE(w) \gg 0).
\end{equation}
We can show that $\zeta_m(w, s)$ is holomorphic at $w=0$. (See Proposition \ref{spectral-zeta}.)

Let us define the regularized determinants of the Laplacian $\Delta_0^{(1)} \big|_{V_m^{(2)}}$.
\begin{definition}[Determinants of restrictions of $\Delta_0^{(1)} $]
Let $m \in 2\N$. For $s \gg 0$, define
\begin{equation}
\mathrm{Det}\Bigl(  \Delta_{0}^{(1)} \big|_{V_{m}^{(2)}}  +s(s-1) \Bigr) 
:= \exp \Bigl(  - \frac{\partial}{\partial w} \Big|_{w=0}  \sum_{n=0}^{\infty} \frac{1}{\bigl( \lambda_n(m)+s(s-1) \bigr)^w} \Bigr). 
\end{equation}
\end{definition}
We see later that $\mathrm{Det}\Bigl(  \Delta_{0}^{(1)} \big|_{V_{m}^{(2)}}  +s(s-1) \Bigr)$
can be extended to an entire function of $s$.
(See Corollary \ref{cor1}.)

Our main theorem is as follows.
\begin{theorem}[Main Theorem] \label{mainth}
Let $\square_m :=  \Delta_{0}^{(1)} \big|_{V_{m}^{(2)}}$ for $m \in 2 \N$. 
We have the following determinant expressions of the completed Selberg type zeta functions.
\begin{enumerate}
\item $\displaystyle{\widehat{Z}_2^{\frac12}(s)=e^{(s-\frac{1}{2})^2 \zeta_K(-1) +C_2}  \, \frac{\mathrm{Det} \bigl( \square_2+s(s-1) \bigr)}{s(s-1)}}$.
\item $\displaystyle{\widehat{Z}_4(s) = e^{2(s-\frac{1}{2})^2 \zeta_K(-1) +C_4}  \,
\frac{ s(s-1) \cdot \mathrm{Det} \bigl( \square_4+s(s-1) \bigr)}{\mathrm{Det} \bigl( \square_2+s(s-1) \bigr)}}$.
\item For $m \ge 6$,  $\displaystyle{\widehat{Z}_m(s) = e^{2(s-\frac{1}{2})^2 \zeta_K(-1) +C_m}  
\frac{\mathrm{Det} \bigl( \square_m+s(s-1) \bigr)}{\mathrm{Det} \bigl( \square_{m-2}+s(s-1) \bigr)}}$.  
\end{enumerate}
Here, the constants $C_m$ are given by
\begin{align*}
   C_2 &= - \frac{1}{2} \log \varepsilon + \sum_{j=1}^{N}\frac{\nu_j^2-1}{12 \nu_j}\log \nu_j, \\
   C_m &=  \sum_{j=1}^{N} \frac{\nu_j^2-1 -12\alpha_0(m,j) \bigl\{ \nu_j - \alpha_0(m,j)  \bigr\} }{6 \nu_j}\log \nu_j \quad  (m \ge 4),
\end{align*}
the natural numbers $\nu_1,\nu_2,\dots,\nu_N$ are the orders of the elliptic fixed points in $X_K$
and the integers $\alpha_0(m,j) \in \{ 0,1,\dots,\nu_j-1\}$ are defined in (\ref{eq:alpha}).
\end{theorem}

We know the following Weyl's law: 
\[ N_m^{+}(T) := \#\{ j \, | \, \lambda_j(m)  \le T  \} \sim \tfrac{(m-1)}{2} \cdot \zeta_K(-1) \cdot T 
\quad (T \to \infty). \]
(See \cite[Theorem 6.11]{G1}.)
Therefore, we may say that
$Z_m(s)$ ($m \ge 4$) have ``more'' zeros than poles. 

We have several corollaries from Theorem \ref{mainth} by direct calculation.
\begin{corollary} \label{cor1}
Let $\square_m =  \Delta_{0}^{(1)} \big|_{V_{m}^{(2)}}$ for $m \in 2 \N$.
For $m \in 2\N$, we have
\begin{enumerate}
\item $\mathrm{Det}\bigl(\square_2+s(s-1)\bigr) = s(s-1) \, e^{-(s-\frac{1}{2})^2 \zeta_K(-1)-C_2} \, \widehat{Z}_2^{\frac12}(s)$. 
\item $\mathrm{Det}\bigl(\square_m+s(s-1)\bigr) = e^{ -(m-1)(s-\frac{1}{2})^2 \zeta_K(-1)  -(C_2+C_4+\cdots +C_m)} 
\, \widehat{Z}_2^{\frac12}(s) \, \widehat{Z}_4(s) \cdots \widehat{Z}_m(s)$ for $m \ge 4$.
\end{enumerate}
\end{corollary}
It follows from the above corollary that $\mathrm{Det}\bigl(\square_m+s(s-1)\bigr)$ $(m \in 2 \N)$
can be extended to entire functions of $s$.

By putting $s=1$ in the above, we have
\begin{corollary} \label{cor2}
For $m \in 2\N$, we have
\begin{enumerate}
\item $\mathrm{Det}(\square_2) = e^{-\frac{1}{4} \zeta_K(-1)-C_2} \, \mathrm{Res}_{s=1} \widehat{Z}_2^{\frac12}(s)$. 
\item $\mathrm{Det}(\square_4) = e^{-\frac{3}{4} \zeta_K(-1)-(C_2+C_4)} \,  
\mathrm{Res}_{s=1} \widehat{Z}_2^{\frac12}(s) \cdot \widehat{Z}_4'(1)$. 
\item  $\mathrm{Det}(\square_m) = e^{ -\frac{m-1}{4} \zeta_K(-1)  -(C_2+C_4+\cdots +C_m)} \,  
\mathrm{Res}_{s=1} \widehat{Z}_2^{\frac12}(s)  \cdot \widehat{Z}_4'(1) \cdot \widehat{Z}_6(1) \cdots \widehat{Z}_m(1)$ \\
for $m \ge 6$.
\end{enumerate}
\end{corollary}
Here, $\square_m =  \Delta_{0}^{(1)} \big|_{V_{m}^{(2)}}$ for $m \in 2 \N$.

\section{Preliminaries} 
We fix the notation for the Hilbert modular group of a real quadratic field
in this section.
We also recall the definition of Hilbert-Maass forms for the Hilbert modular group
and review ``Differences of the Selberg trace formula'', 
introduced in \cite{G1},
which play a crucial role in this article.

\subsection{Hilbert modular group of a real quadratic field} 
Let
$K/\Q$ be a real quadratic field with class number one and
$\mathcal{O}_K$ be the ring of integers of $K$. Put $D$ be the
discriminant of $K$ and $\varepsilon > 1 $ be the fundamental
unit of $K$. 
We denote the generator of $\mathrm{Gal}(K/\Q)$
by $\sigma$ and put $a' := \sigma(a)$ 
for $a \in K$. 
We also put
$\gamma' =  \Bigl( 
\begin{array}{cc}
a' & b' \\
c' & d'
\end{array} \Bigr)$ for 
$\gamma =  \Bigl( 
\begin{array}{cc}
a & b \\
c & d
\end{array} \Bigr) \in \mathrm{PSL}(2, \mathcal{O}_K)$.

Let $G$ be $\mathrm{PSL}(2,\R)^2 
= \Bigl( \mathrm{SL}(2,\R) / \{ \pm I \} \Bigr)^2$ and 
$\bH^2$ be the direct product of two copies of the upper half plane 
$\bH := \{ z \in \C \, | \,  \IM(z)>0  \}$.
The group $G$ acts on $\bH^2$ by 
\[ g.z = (g_1,g_2).(z_1,z_2) = 
\biggl( \frac{a_1z_1+b_1}{c_1z_1+d_1},\frac{a_2z_2+b_2}{c_2z_2+d_2} \biggr) \in \bH^2
\]
for $g=(g_1,g_2)=
( \Bigl( 
\begin{array}{cc}
a_1 & b_1 \\
c_1 & d_1
\end{array} \Bigr),
 \Bigl( 
\begin{array}{cc}
a_2 & b_2 \\
c_2 & d_2
\end{array} \Bigr)
)$ and $z=(z_1,z_2) \in \bH^2$.

A discrete subgroup $\Gamma \subset G$ is called irreducible if 
it is not commensurable with any direct product $\Gamma_1 \times \Gamma_2$ 
of two discrete subgroups of $\mathrm{PSL}(2,\R)$. 
We have classification of the elements of irreducible $\Gamma$.
\begin{proposition}[Classification of the elements]
Let $\Gamma$ be an irreducible discrete subgroup of $G$. 
Then any element of $\Gamma$ is one of the followings. 
\begin{enumerate}
\item $\gamma=(I,I)$ is the identity 
\item $\gamma=(\gamma_1,\gamma_2)$ is hyperbolic \, $\Lra \, 
|\tr(\gamma_1)| > 2$ and $|\tr(\gamma_2)| > 2$  
\item $\gamma=(\gamma_1,\gamma_2)$ is elliptic \, $\Lra \, 
|\tr(\gamma_1)| < 2$ and $|\tr(\gamma_2)| < 2$
\item $\gamma=(\gamma_1,\gamma_2)$ is hyperbolic-elliptic \, $\Lra \, 
|\tr(\gamma_1)| > 2$ and $|\tr(\gamma_2)| < 2$  
\item $\gamma=(\gamma_1,\gamma_2)$ is elliptic-hyperbolic \, $\Lra \, 
|\tr(\gamma_1)| < 2$ and $|\tr(\gamma_2)| > 2$  
\item $\gamma=(\gamma_1,\gamma_2)$ is parabolic \, $\Lra \, 
|\tr(\gamma_1)| = |\tr(\gamma_2)| = 2$ 
\end{enumerate}
\end{proposition}
Note that there are no other types in $\Gamma$. (parabolic-elliptic etc.)

Let us consider the Hilbert modular group of the real quadratic field $K$
with class number one,
\[ \Gamma_{K} := 
\Bigl\{ (\gamma,\gamma') =  \Bigl( 
\Bigl( 
\begin{array}{cc}
a & b \\
c & d
\end{array} \Bigr), 
\, 
\Bigl( 
\begin{array}{cc}
a' & b' \\
c' & d'
\end{array} \Bigr)
\Bigr)
\Big| \, 
\Bigl( 
\begin{array}{cc}
a & b \\
c & d
\end{array} \Bigr) \in 
\mathrm{PSL}(2,\mathcal{O}_K)
\Big\}. 
\]

It is known that $\Gamma_K$ is an irreducible discrete subgroup 
of $G=\mathrm{PSL}(2,\R)^2$ with the only one cusp 
$\infty := (\infty, \infty)$,
i.e. $\Gamma_K$-inequivalent parabolic fixed point.
$X_K = \Gamma_K \backslash \bH^2$ 
is called the Hilbert modular surface.

We have a lemma about the Euler characteristic of the Hilbert modular surface $X_K$.
\begin{lemma} \label{EC}
Let 
$E(X_K)$ be the Euler characteristic of the Hilbert modular surface $X_K = \Gamma_K \backslash \bH^2$.
Then we have $E(X_{K}) \in 2\N$.
\end{lemma}
\begin{proof}
By noting the formula
$E(X_K) = 2 \zeta_{K}(-1)  +  \sum_{j=1}^{N} \frac{\nu_j-1}{\nu_j}$
(see (2), (4) on \cite[pp.46-47]{HZ}), 
$E(X_K)$ is a positive integer. 
Let $Y_K$ and $Y_K^{-}$ be the non-singular algebraic surfaces
resolved singularities, in the canonical minimal way, 
of compactifications of $\Gamma_K \backslash \bH^2$ and
$\Gamma_K \backslash (\bH \times \bH^{-})$ respectively. Here $\bH^{-}$ is the lower half plane.
Let $\chi(Y_K)$ and $\chi(Y_K^{-})$ be the arithmetic genera of $Y_K$ and $Y_K^{-}$ respectively.
By the formulas (12) and (14) on \cite[p.48]{HZ}, we have
\[ E(X_K) = 2 \bigl( \chi(Y_K)+\chi(Y_K^{-}) \bigr). \]
We complete the proof.
\end{proof}

We fix the notation for elliptic conjugacy classes in $\Gamma_K$.
Let $R_1,R_2,\cdots,R_{N}$ be a complete system of representatives
of the $\Gamma_K$-conjugacy classes of primitive elliptic elements of
$\Gamma_K$. $\nu_1,\nu_2, \cdots, \nu_N$ $(\nu_j \in \N, \, \nu_j \ge 2)$
denote the orders of $R_1,R_2,\cdots,R_{N}$.
We may assume that $R_j$ is conjugate 
in $\mathrm{PSL}(2,\R)^2$ to
\[ R_j \sim \Bigl( 
\Bigl( 
\begin{array}{cc}
\cos \frac{\pi}{\nu_j} & - \sin \frac{\pi}{\nu_j} \\
\sin \frac{\pi}{\nu_j} & \cos \frac{\pi}{\nu_j}
\end{array} \Bigr), 
\, 
\Bigl( 
\begin{array}{cc}
\cos \frac{t_j \pi}{\nu_j} & - \sin \frac{t_j \pi}{\nu_j} \\
\sin \frac{t_j \pi}{\nu_j} & \cos \frac{t_j \pi}{\nu_j}
\end{array} \Bigr)\Bigr), 
\quad (t_j,\nu_j)=1. 
\]
For even natural number $m \ge 2$ and 
$l \in \{0,1,\cdots, \nu_j-1\}$, we define 
$\alpha_l(m,j), \, \overline{\alpha_l}(m,j) \in \{0,1, \cdots, \nu_j -1 \}$
by
\begin{equation} \label{eq:alpha}
\begin{split}
& l + \frac{t_j (m-2)}{2} \equiv \alpha_l(m,j) \pmod{\nu_j}, \\
& l - \frac{t_j (m-2)}{2} \equiv \overline{\alpha_l}(m,j) 
\pmod{\nu_j}. \\
\end{split}
\end{equation}

We divide hyperbolic conjugacy classes of $\Gamma_K$ into two subclasses according to their types.
\begin{definition}[Types of hyperbolic elements] \label{d:type}
For a hyperbolic element $\gamma$, we define that
\begin{enumerate}
\item $\gamma$ is type 1 hyperbolic $\Lra$ whose all fixed points are 
not fixed by parabolic elements.
\item $\gamma$ is type 2 hyperbolic $\Lra$ not type 1 hyperbolic.
\end{enumerate}  
\end{definition}

We denote by $\Gamma_{\mathrm{H1}}$, 
$\Gamma_{\mathrm{E}}$, 
$\Gamma_{\mathrm{HE}}$,
$\Gamma_{\mathrm{EH}}$ and
$\Gamma_{\mathrm{H2}}$,  
type 1 hyperbolic $\Gamma_K$-conjugacy classes,
elliptic $\Gamma_K$-conjugacy classes,
hyperbolic-elliptic $\Gamma_K$-conjugacy classes, 
elliptic-hyperbolic $\Gamma_K$-conjugacy classes and
type 2 hyperbolic $\Gamma_K$-conjugacy classes
of $\Gamma_K$ respectively.

\subsection{The space of Hilbert-Maass forms}

Fix the weight $(m_1,m_2) \in (2 \Z)^2$. 
Set the automorphic factor $j_{\gamma}(z_j) = \frac{cz_j+d}{|cz_j+d|}$ for 
$\gamma \in \mathrm{PSL}(2,\R)$ $(j=1,2)$. 

Let $\Delta_{m_j}^{(j)} := -y_j^2(\frac{\partial^2}{\partial x_j^2}
   +\frac{\partial^2}{\partial y_j^2}) 
+ i  m_j \, y_j \frac{\partial}{\partial x_j} \quad (j=1,2)$
be the Laplacians of weight $m_j$ for the variable $z_j$.

Let us define the $L^2$-space of 
automorphic forms of weight $(m_1,m_2)$ with respect to the 
Hilbert modular group $\Gamma_{K}$.
\begin{definition}[$L^2$-space of automorphic forms of weight $(m_1,m_2)$]
\begin{eqnarray*}
 &&  L^2(\Gamma_K \backslash \bH^2 \, ; \, (m_1,m_2)) :=  
\Bigl\{ f \colon \bH^2 \to \C, \, C^{\infty} \, \Big| \, \\
&&  (i) \, 
f((\gamma,\gamma')(z_1,z_2)) 
= j_{\gamma}(z_1)^{m_1} j_{\gamma'}(z_2)^{m_2}f(z_1,z_2) 
\quad \forall (\gamma,\gamma') \in \Gamma_K \\
&&  (ii) \, 
\exists (\lambda^{(1)},\lambda^{(2)}) \in \R^2 \quad 
\Delta_{m_1}^{(1)} \, f(z_1,z_2) = \lambda^{(1)} f(z_1,z_2), \quad 
\Delta_{m_2}^{(2)} \, f(z_1,z_2) = \lambda^{(2)} f(z_1,z_2) \\
&& (iii) \, ||f ||^2 = \int_{\Gamma_K \backslash \bH^2} f(z) 
\overline{f(z)} \, 
d \mu (z) < \infty
\Bigr\}.
\end{eqnarray*}
Here, $d \mu(z) = \frac{dx_1dy_1}{y_1^2} \frac{dx_2dy_2}{y_2^2}$
for $z=(z_1,z_2) \in \bH^2$.
\end{definition}

Then, it is known that
\begin{proposition}
Let $L^2_{\text{dis}}(\Gamma_K \backslash \bH^2 \, ; \, (m_1,m_2))$
be the subspace of the discrete spectrum of the Laplacians
and 
$L^2_{\text{con}}(\Gamma_K \backslash \bH^2 \, ; \, (m_1,m_2))$
be the subspace of the continuous spectrum.
Then, 
we have a direct sum decomposition :  
\[  L^2(\Gamma_K \backslash \bH^2 \, ; \, (m_1,m_2)) 
 = L^2_{\text{dis}}(\Gamma_K \backslash \bH^2 \, ; \, (m_1,m_2))
   \oplus L^2_{\text{con}}(\Gamma_K \backslash \bH^2 \, ; \, (m_1,m_2))  \]
and there is an orthonormal basis $\{ \phi_j \}_{j=0}^{\infty}$
of $L^2_{\text{dis}}(\Gamma_K \backslash \bH^2 \, ; \, (m_1,m_2))$.
\end{proposition}

\begin{definition}[Hilbert Maass forms of weight $(m_1,m_2)$]
\label{def:HM}
Let $(m_1,m_2) \in (2\Z)^2$.
We call   
\[ L^2_{\text{dis}}(\Gamma_K \backslash \bH^2 \, ; \, (m_1,m_2)) \]
the space of Hilbert Maass forms for $\Gamma_K$ of weight $(m_1,m_2)$.
\end{definition}

Let $\{ \phi_{j} \}_{j=0}^{\infty}$ be an orthonormal basis
of $L^2_{\text{dis}}(\Gamma_{K} \backslash \bH^2 
\, ; \, (m_1,m_2))$ and 
$(\lambda_j^{(1)},\lambda_j^{(2)}) \in \R^2$ such that
\[ \Delta_{m_1}^{(1)} \phi_j = \lambda_j^{(1)} \phi_j \quad \mbox{and} \quad
   \Delta_{m_2}^{(2)} \phi_j = \lambda_j^{(2)} \phi_j. \]

We write $\lambda^{(l)}_j = \tfrac{1}{4}+(r^{(l)}_j)^2$ and 
$r_j^{(i)}$ are defined by
\begin{equation} \label{eq:rj}
r_j^{(l)} := 
\begin{cases}
  \sqrt{\lambda_j^{(l)}- \frac{1}{4}}     \quad \: \, \mbox{if } \lambda_j^{(l)} \ge \frac{1}{4},  \\
 i \sqrt{\frac{1}{4} - \lambda_j^{(l)}}   \quad \mbox{if } \lambda_j^{(l)} < \frac{1}{4}, 
\end{cases}
\end{equation}
for $l=1,2$.

\subsection{Double differences of the Selberg trace formula}
Let $m$ be an even integer. We studied and derived the full Selberg trace formula for 
$L^2(\Gamma_K \backslash \bH^2 \, ; \, (0,m))$ in \cite{G1}.
(See \cite[Theorem 2.22]{G1}.) 
Let $h(r_1,r_2)$ be an even ``test function'' which satisfy certain analytic conditions. 
Roughly speaking, \cite[Theorem 2.22]{G1} is as follows.
\[ \sum_{j=0}^{\infty} h(r_j^{(1)},r_j^{(2)}) 
= \mathbf{I}(h)+\mathbf{II_a}(h)+\mathbf{II_b}(h)+\mathbf{III}(h).
\]
Here, the right hand side is a sum of distributions of $h$ 
contributed from several conjugacy classes of $\Gamma_K$ 
and Eisenstein series for $\Gamma_K$.
Assuming that the test function $h(r_1,r_2)$ is a product of 
$h_1(r_1)$ and $h_2(r_2)$, we derived ``differences of STF''(\cite[Theorem 4.1]{G1})
and ``double differences of STF'' (\cite[Theorem 4.4]{G1}). We explain for this.

Let us consider the subspace of $L^2_\text{dis} ( \Gamma_{K} \backslash \mathbb{H}^2 
; (0,m) )$ given by
\[ V_{m}^{(2)} = 
\Bigl\{  f \in  
L^2_\text{dis} ( \Gamma_{K} \backslash \mathbb{H}^2 
; (0,m) )
\Bigl| \,  \Delta_{m}^{(2)} f  = 
\frac{m}{2} \bigl( 1-\frac{m}{2} \bigr) \, f  \Bigr\}.
\]

Let $h_1(r)$ be an even function, analytic in $\IM(r)<\delta$ for some $\delta>0$,  
\[ h_1(r)=O((1+|r|^2)^{-2-\delta}) \]
for some $\delta>0$ in this domain.
Let $g_1(u):= \frac{1}{2 \pi}\int_{-\infty}^{\infty} h_1(r) e^{-i r u} \, dr$. 
Then we have
\begin{proposition}[Double differences of STF 
for $L^2 \bigl( \Gamma_{K} \backslash \bH^2 \, ; \, (0,2) \bigr)$]
\label{ddtrf2}
Let $m = 2$. We have 
\begin{align*}
& \sum_{j=0}^{\infty} h_1 \Bigl( \rho_j(2) \Bigr)  
- h_1 \Bigl( \frac{i}{2} \Bigr) \\
 & \quad =  \:
\frac{\mathrm{vol}(\Gamma_K \backslash \bH^2)}{16 \pi^2}
\int_{-\infty}^{\infty}
  r h_1(r) \tanh (\pi r)  \, dr  \\
& \quad \quad \quad - \sum_{R(\theta_1,\theta_2) \in 
\Gamma_{\mathrm{E}}} 
\frac{i e^{ - i \theta_1}}{8 \nu_{R} \sin \theta_1}
\int_{-\infty}^{\infty} g_1(u) \, e^{-u/2} \biggl[  \frac{e^u - e^{2 i \theta_1}}{\cosh u - \cos 2 \theta_1} \biggr] 
du \\
& \quad \quad \quad
 - \frac{1}{2} \sum_{(\gamma,\omega) \in \Gamma_{\mathrm{HE}}} 
\frac{\log N(\gamma_0) \, g_1(\log N(\gamma))   }{N(\gamma)^{1/2}-N(\gamma)^{-1/2}} 
- \log \varepsilon \, g_1(0)
\, - 2 \log \varepsilon \, 
\sum_{k=1}^{\infty} g_1(2k \log \varepsilon) \, \varepsilon^{-k}. 
\end{align*}
Here,  $\{ \lambda_j(2) = 1/4+ \rho_j(2)^2 \}_{j=0}^{\infty}$  
is the set of eigenvalues of the Laplacian $\Delta_{0}^{(1)}$ acting 
on $V_2^{(2)}$.
\end{proposition}
\begin{proof}
See \cite[Corollary 6.3]{G1}.
\end{proof}

\begin{proposition}[Double differences of STF 
for $L^2 \bigl( \Gamma_{K} \backslash \bH^2 \, ; \, (0,m) \bigr)$]
\label{ddtrfm}
Let $m \in 2 \N$ and $m \ge 4$. We have 
\begin{align*}
& \sum_{j=0}^{\infty} h_1 \Bigl( \rho_j(m) \Bigr)  
   - \sum_{j=0}^{\infty} h_1 \Bigl( \rho_j(m-2) \Bigr)  
   + \delta_{m,4} \, h_1 \Bigl( \frac{i}{2} \Bigr) \\
 & \quad =  \:
\frac{\mathrm{vol}(\Gamma_K \backslash \bH^2)}{8 \pi^2}
\int_{-\infty}^{\infty}
  r h_1(r) \tanh (\pi r)  \, dr \\
& \quad \quad \quad - \sum_{R(\theta_1,\theta_2) \in 
\Gamma_{\mathrm{E}}} 
\frac{i e^{ - i \theta_1} \, e^{i(m-2) \theta_2}}{4 \nu_{R} \sin \theta_1}
\int_{-\infty}^{\infty} g_1(u) \, e^{-u/2} \biggl[  \frac{e^u - e^{2 i \theta_1}}{\cosh u - \cos 2 \theta_1} \biggr] 
du \\
& \quad \quad \quad
 - \sum_{(\gamma,\omega) \in \Gamma_{\mathrm{HE}}} 
\frac{\log N(\gamma_0)}{N(\gamma)^{1/2}-N(\gamma)^{-1/2}} \, g_1(\log N(\gamma)) 
\, e^{i(m-2)\omega}
\\
& \quad \quad \quad - 2 \log \varepsilon \, 
\sum_{k=1}^{\infty} g_1(2k \log \varepsilon) \, 
 \Bigl( \varepsilon^{-k(m-1)}  - \varepsilon^{-k(m-3)} \Bigr). 
\end{align*}
\end{proposition}
Here,  $\{ \lambda_j(q) = 1/4+ \rho_j(q)^2 \}_{j=0}^{\infty}$  
is the set of eigenvalues of the Laplacian $\Delta_{0}^{(1)}$ acting 
on $V_q^{(2)}$ $(q=m,m-2)$.
\begin{proof}
See \cite[Theorem 4.4]{G1} and \cite[(5.3)]{G1}.

\end{proof}

\section{Asymptotic behavior of the completed Selberg zeta functions} 

We have to know the asymptotic behavior of the completed Selberg zeta functions
$\widehat{Z}_{2}^{\frac12}(s)$ and $\widehat{Z}_{m}(s)$ $(m \ge 4)$ when $s \to \infty$, 
to prove Main Theorem (Theorem \ref{mainth}). We calculate their asymptotic behavior 
in this section.

\begin{lemma}[Stirling's formula for $\Gamma_2(z)$] \label{gamma2}
Let $\Gamma_2(z) :=\exp \bigl( \frac{\partial}{\partial s} \big|_{s=0} \sum_{m,n=0}^{\infty} (m+n+z)^{-s} \bigr)$ be the
double Gamma function. Then we have
\begin{equation}
\log \Gamma_2(z+1)  = \frac{3}{4} z^2 - \Bigl( \frac{z^2}{2}-\frac{1}{12} \Bigr) \log z +o(1) \quad (z \to \infty). 
\end{equation}
\end{lemma}
\begin{proof}
Let $G(z)$ be the  Barnes $G$-function defined by (See \cite[p.268]{B1}.)
\[ G(z+1) = (2 \pi)^{\frac{z}{2}} e^{-\frac{z+z^2(1+\gamma)}{2}}
\prod_{k=1}^{\infty} \biggl\{ \Bigl(1+\frac{z}{k} \Bigr)^k e^{ -z+\frac{z^2}{k^2}} \biggr\}. \]
Here, $\gamma=-\Gamma'(1)$ is the Euler constant. 
By using the relation (See \cite[Proposition 4.1]{On}.)
\[ \Gamma_2(z) = e^{\zeta'(-1)}(2 \pi) ^{\frac{z-1}{2}} G(z)^{-1}, \]
and the asymptotic formula (See \cite[p.269]{B1}.)
\[ \log G(z+1) =   \frac{z}{2} \log(2 \pi) +\zeta'(-1) -\frac{3}{4} z^2 + \Bigl( \frac{z^2}{2}-\frac{1}{12} \Bigr) \log z +o(1) \quad (z \to \infty), \]
we have the desired formula.
\end{proof}

\begin{lemma}[Asymptotics of the identity factors] \label{z-id}
We have
\begin{equation} \label{z-id-half}
\log Z_{\mathrm{id}}^{\frac12} (s)  = \zeta_{K}(-1) \biggl\{ \frac{3}{2} s^2 - s -\Bigl( s^2-s+\frac{1}{3} \Bigr) \log s \biggr\} +o(1) \quad (s \to \infty),
\end{equation}
\begin{equation} \label{z-id-m}
\log Z_{\mathrm{id}} (s)  = 2 \zeta_{K}(-1) \biggl\{ \frac{3}{2} s^2 - s -\Bigl( s^2-s+\frac{1}{3} \Bigr) \log s \biggr\} +o(1) \quad (s \to \infty). 
\end{equation}
\end{lemma}
\begin{proof}
By Definition \ref{zhat},
\[ \log Z_{\mathrm{id}}^{\frac12} (s)  = \zeta_{K}(-1) \Bigl( \log \Gamma_2(s) +\log\Gamma_2(s+1) \Bigr) \]
and Lemma \ref{gamma2}, we have the desired (\ref{z-id-half}). We see that the relation 
$\log Z_{\mathrm{id}}(s) = 2\log Z_{\mathrm{id}}^{\frac12} (s)$ implies (\ref{z-id-m}). It completes the proof.
\end{proof}

\begin{lemma}[Asymptotics of the elliptic factors] \label{z-ell}
We have
\begin{equation} \label{z-ell-half}
\log Z_{\mathrm{ell}}^{\frac12} (s;2)   =  - \sum_{j=1}^{N} \frac{\nu_j^2-1}{12 \nu_j} \log \frac{s}{\nu_j} +o(1) \quad (s \to \infty),
\end{equation}
\begin{equation}  \label{z-ell-m}
\log Z_{\mathrm{ell}} (s;m)   
=  - \sum_{j=1}^{N} \frac{\nu_j^2-1-12 \alpha_0(m,j) \bigl\{ \nu_j-\alpha_0(m,j) \bigr\}  }{6 \nu_j} \log \frac{s}{\nu_j} +o(1) \quad (s \to \infty)
\end{equation}
for $m \in 2 \N$ and $m \ge 4$.
Here $\alpha_0(m,j) \in \{ 0,1,\dots,\nu_j-1\}$ are defined in (\ref{eq:alpha}).
\end{lemma}
\begin{proof}
We use Stirling's formula of $\Gamma(z)$. (See \cite[p.12]{L}.) 
\[ \log \Gamma(z) = \Bigl( z-\frac{1}{2} \Bigr)  \log z   -z  +\frac{1}{2} \log (2 \pi) +o(1)  \quad (z \to \infty).  \]
By Definition \ref{zhat},
\[ \log Z_{\mathrm{ell}}(s;m)
= \sum_{j=1}^{N} \sum_{l=0}^{\nu_j-1} \frac{\nu_j-1- \alpha_{l}(m,j)  - \overline{\alpha_{l}}(m,j) }{ \nu_j} \log \Gamma \Bigl( \frac{s+l}{\nu_j} \Bigr).
\]
We see that $\{ \alpha_l(m,j) \mid 0 \le l \le \nu_j-1 \}=\{ \overline \alpha_l(m,j) \mid 0 \le l \le \nu_j-1 \} =\{  0,1,2,\dots,\nu_j-1 \}$
for each $j$. 
Thus we have $\sum_{l=0}^{\nu_j-1} \bigl( \nu_j-1- \alpha_{l}(m,j)- \overline{\alpha_{l}}(m,j) \bigr)=0$, and find that
\begin{align*}
& \sum_{l=0}^{\nu_j-1} \frac{\nu_j-1- \alpha_{l}(m,j)  - \overline{\alpha_{l}}(m,j) }{\nu_j} \log \Gamma \Bigl( \frac{s+l}{\nu_j} \Bigr) \\
&= \sum_{l=0}^{\nu_j-1} \frac{\nu_j-1-\alpha_{l}(m,j)  - \overline{\alpha_{l}}(m,j) }{\nu_j} 
\biggl\{ \Bigl( \frac{s+l}{\nu_j} -\frac{1}{2} \Bigr) \log\Bigl( \frac{s+l}{\nu_j} \Bigr)    
    - \frac{s+l}{\nu_j}  +\frac{1}{2} \log (2 \pi)  \biggr\} +o(1)  \\
&= \sum_{l=0}^{\nu_j-1} \frac{\nu-1-\alpha_{l}(m,j)  - \overline{\alpha_{l}}(m,j)}{\nu_j}
\biggl\{ \Bigl( \frac{s+l}{\nu_j} -\frac{1}{2} \Bigr) \log (s+l)  
- \frac{l}{\nu} \log \nu_j - \frac{l}{\nu_j}  \biggr\} +o(1) \\
&= \sum_{l=0}^{\nu_j-1} \frac{\nu_j-1-\alpha_{l}(m,j)  - \overline{\alpha_{l}}(m,j)}{\nu_j}
\biggl\{ \Bigl( \frac{s}{\nu_j} -\frac{1}{2} \Bigr) \log s  
+ \frac{l}{\nu_j} \log \frac{s}{\nu_j}   \biggr\} +o(1) \\
&= \sum_{l=0}^{\nu_j-1} \frac{\nu_j-1-\alpha_{l}(m,j)  - \overline{\alpha_{l}}(m,j)}{\nu_j}  
\cdot \frac{l}{\nu_j} \log \frac{s}{\nu_j}  +o(1) \\
&= \frac{(\nu_j-1)^2}{2 \nu_j} \log \frac{s}{\nu_j} 
- \sum_{l=0}^{\nu_j-1} \frac{\alpha_{l}(m,j)  + \overline{\alpha_{l}}(m,j)}{\nu_j}  
\cdot \frac{l}{\nu_j} \log \frac{s}{\nu_j} 
+o(1) \quad (s \to \infty).
\end{align*}
By (\ref{eq:alpha}), we can check that
\[ \alpha_l(m,j)= \begin{cases}
\alpha_0(m,j) +l &  (0 \le l \le \nu_j-\alpha_0(m,j)-1) \\
\alpha_0(m,j)-\nu_j+l & (\nu_j-\alpha_0(m,j) \le l \le \nu_j-1)
\end{cases}
,\]
hence we calculate further,
\begin{align*}
\sum_{l=0}^{\nu_j-1} \frac{l \, \alpha_{l}(m,j)  }{\nu_j^2} 
&=\sum_{l=0}^{\nu_j-\alpha_0(m,j)-1} \frac{l \bigl(  \alpha_0(m,j) +l \bigr)}{\nu_j^2} 
   +\sum_{l=\nu_j-\alpha_0(m,j)}^{\nu_j-1}  \frac{l \bigl(  \alpha_0(m,j) -\nu_j+l \bigr)}{\nu_j^2} \\
&= \frac{(\nu_j-1)(2 \nu_j-1)}{6 \nu_j}  + \frac{ \alpha_0(m,j) \bigl(  \alpha_0(m,j) -\nu_j \bigr) }{\nu_j}.
\end{align*}
By noting $\alpha_0(m,j) \bigl(  \alpha_0(m,j) -\nu_j \bigr)=\overline{\alpha_0}(m,j) \bigl(  \overline{\alpha_0}(m,j) -\nu_j \bigr)$, we have 
\begin{align*}
& \sum_{l=0}^{\nu_j-1} \frac{\nu_j-1- \alpha_{l}(m,j)  - \overline{\alpha_{l}}(m,j) }{\nu_j} \log \Gamma \Bigl( \frac{s+l}{\nu_j} \Bigr) \\
&= \frac{(\nu_j-1)^2}{2 \nu_j} \log \frac{s}{\nu_j} 
- \sum_{l=0}^{\nu_j-1} \frac{\alpha_{l}(m,j)  + \overline{\alpha_{l}}(m,j)}{\nu_j}  
\cdot \frac{l}{\nu_j} \log \frac{s}{\nu_j} +o(1) \\
&= \frac{(\nu_j-1)^2}{2 \nu_j} \log \frac{s}{\nu_j} 
-2 \Bigl\{ \frac{(\nu_j-1)(2 \nu_j-1)}{6 \nu_j}  + \frac{ \alpha_0(m,j) \bigl(  \alpha_0(m,j) -\nu_j \bigr) }{\nu_j} \Bigr\}  \log \frac{s}{\nu_j}
+o(1) \\
&= -\frac{\nu_j^2-1-12 \alpha_0(m,j) \bigl\{ \nu_j-\alpha_0(m,j) \bigr\}  }{6 \nu_j} \log \frac{s}{\nu_j} 
+o(1) \quad (s \to \infty).
\end{align*}
Thus we have (\ref{z-ell-m}).
In addition, we note that
\[ \log Z_{\mathrm{ell}}^{\frac12} (s;2) = \frac{1}{2} \log Z_{\mathrm{ell}}(s;m) \Big|_{m=2}. 
\]
Since $\alpha_l(2,j)=l$, we  see that $\alpha_0(2,j)=0$ for any $j$. Therefore we have   (\ref{z-ell-half}).
It completes the proof.
\end{proof}

\begin{proposition}[Asymptotics of the completed Selberg zeta functions]
We have
\begin{equation} \label{z2-asym}
\begin{split}
\log \widehat{Z}_{2}^{\frac12}(s) =&  \zeta_{K}(-1) \biggl\{ \frac{3}{2} s^2 
- s -\Bigl( s^2-s+\frac{1}{3} \Bigr) \log s \biggr\} \\
&- \sum_{j=1}^{N} \frac{\nu_j^2-1}{12 \nu_j} \log \frac{s}{\nu_j} 
-s \log \varepsilon
+o(1) \quad (s \to \infty),
\end{split}
\end{equation}
\begin{equation} \label{zm-asym}
\begin{split}
\log \widehat{Z}_{m}(s) =&  2 \zeta_{K}(-1) \biggl\{ \frac{3}{2} s^2 
- s -\Bigl( s^2-s+\frac{1}{3} \Bigr) \log s \biggr\} \\
&- \sum_{j=1}^{N} \frac{\nu_j^2-1-12 \alpha_0(m,j) \bigl\{ \nu_j-\alpha_0(m,j) \bigr\}  }{6 \nu_j} \log \frac{s}{\nu_j} 
+o(1) \quad (s \to \infty),
\end{split}
\end{equation}
for $m \in 2 \N$ and $m \ge 4$.
Here $\alpha_0(m,j) \in \{ 0,1,\dots,\nu_j-1\}$ are defined in (\ref{eq:alpha}).
\end{proposition}
\begin{proof}
We note that $\log  \sqrt{Z_{2}(s)}, \, \log Z_m(s)=o(1)$ $(s \to \infty)$.
By Definition \ref{zhat} and Lemmas \ref{z-id} and \ref{z-ell}, we complete the proof. 
\end{proof}

\section{Asymptotic behavior of the regularized determinants} 

To investigate the analytic nature of the spectral zeta function $\zeta_m(w,s)$ 
at $w=0$, 
we introduce the theta function $\theta_m(t)$ in this section.
Since the regularized determinants of the Laplacians 
$\mathrm{Det} \bigl( \square_m+s(s-1) \bigr)$ are defined 
by the derivative of $-\zeta_m(w,s)$ at $w=0$, we need to know
the asymptotics of $-\frac{\partial}{\partial w} \zeta_m(w,s) \big|_{w=0}$ when $s \to \infty$.
We calculate their asymptotics in this section. 
\begin{definition}
For $m \in 2 \N$ and $t>0$, define
\begin{equation}
\theta_m(t) := \sum_{j=0}^{\infty} e^{-t \, \lambda_j(m)}.
\end{equation}
\end{definition}

We investigate the asymptotic behavior of $\theta_m(t)$ as $t \to +0$ by using 
Propositions \ref{ddtrf2} and \ref{ddtrfm}, which are called 
``Double differences of the Selberg trace formula for Hilbert modular surfaces'' 
introduced and proved in \cite{G1}.

\begin{proposition} \label{theta-asym}
We have the following asymptotic formulas.
\begin{equation} \label{theta2}
\theta_2(t) =   \frac{1}{2} \zeta_{K}(-1) \frac{1}{t} -\frac{\log \varepsilon}{2 \sqrt{\pi}}\frac{1}{\sqrt{t}} 
+ \Bigl ( -\frac{1}{6} \zeta_{K}(-1) +b_0(2)+1 \Bigr) +o(1) \quad (t \to +0), 
\end{equation}

\begin{equation} \label{thetam}
\begin{split}
\theta_m(t)  = &  \frac{m-1}{2} \zeta_{K}(-1) \, \frac{1}{t} -\frac{\log \varepsilon}{2 \sqrt{\pi}}\frac{1}{\sqrt{t}} 
+ \Bigl ( -\frac{m-1}{6} \zeta_{K}(-1) + b_0(2)+b_0(4)+\cdots+b_0(m) \Bigr) \\
& +o(1) \quad (t \to +0), \quad (m \in 2 \N, \, m \ge 4).
\end{split}
\end{equation}
Here, $\displaystyle{b_0(2)=- \sum_{j=1}^{N} \frac{\nu_j^2-1}{24 \nu_j}}$, 
$\displaystyle{b_0(m)=-\sum_{j=1}^{N} \frac{\nu_j^2-1 -12\alpha_0(m,j) \bigl\{ \nu_j - \alpha_0(m,j)  \bigr\} }{12 \nu_j}}$ $(m \ge 4)$.
\end{proposition}
\begin{proof}
For $t>0$, let us take the pair of test functions $h_1(r)=e^{-t(r^2+1/4)}$ and 
$g_1(u) = \frac{1}{\sqrt{4 \pi t}} \exp \bigl( -\frac{t}{4} -\frac{u^2}{4t} \bigr)$
in Proposition \ref{ddtrf2}, then we have
\begin{equation} \label{theta2-tr}
\theta_2(t) - 1
= I_2(t)+E_2(t)+ HE_2(t)+PS_2(t)+HS_2(t). 
\end{equation}
Here, 
\begin{itemize}
\item $I_2(t)=\frac{\mathrm{vol}(\Gamma_K \backslash \bH^2)}{16 \pi^2}
\int_{-\infty}^{\infty}
   \exp \bigl( -t(r^2+1/4) \bigr)  \, r \tanh (\pi r)  \, dr$,
\item $E_2(t)= - \sum_{R(\theta_1,\theta_2) \in 
\Gamma_{\mathrm{E}}} 
\frac{i e^{ - i \theta_1}}{8 \nu_{R} \sin \theta_1}
\int_{-\infty}^{\infty}   \frac{1}{\sqrt{4 \pi t}} \exp \bigl( -\frac{t}{4} -\frac{u^2}{4t} \bigr) \, 
e^{-u/2} \bigl[  \frac{e^u - e^{2 i \theta_1}}{\cosh u - \cos 2 \theta_1} \bigr] du$, 
\item $HE_2(t) = - \frac{1}{2} \sum_{(\gamma,\omega) \in \Gamma_{\mathrm{HE}}} 
\frac{\log N(\gamma_0) }{N(\gamma)^{1/2}-N(\gamma)^{-1/2}} \,  \frac{1}{\sqrt{4 \pi t}} \exp \bigl( -\frac{t}{4} -\frac{(\log N(\gamma))^2}{4t} \bigr)$,
\item $PS_2(t)=- \log \varepsilon \, \frac{1}{\sqrt{4 \pi t}} \exp \bigl( -\frac{t}{4} \bigr)$,
\item $HS_2(t)=- 2 \log \varepsilon 
\sum_{k=1}^{\infty}  \frac{1}{\sqrt{4 \pi t}} \exp \bigl( -\frac{t}{4} -\frac{(2k \log \varepsilon)^2}{4t} \bigr)  \, \varepsilon^{-k}$.
\end{itemize}
Firstly, we see that $HE_2(t)$ and $HS_2(t)$ are exponentially decreasing as $t \to +0$.   
Secondly, by changing the variable $u$ to $\sqrt{t}u$ in $E_2(t)$, we see that there is a constant $b_0(2)$
such that $E_2(t)=b_0(2)+o(1)$ $(t \to +0)$. Thirdly, $PS_2(t)=- \log \varepsilon \, \frac{1}{\sqrt{4 \pi t}} \bigl( 1 - t/4 +o(t) \bigr)$ $(t \to +0)$.
Lastly, noting $\frac{\mathrm{vol}(\Gamma_K \backslash \bH^2)}{8\pi^2} = \zeta_{K}(-1)$ and integration by parts, we have
\begin{align*}
I_2(t) = & \frac{1}{2} \zeta_{K}(-1) \, \frac{1}{2t} \int_{-\infty}^{\infty} \exp \biggl( -t\Bigl( r^2+\frac{1}{4} \Bigr) \biggr) \frac{\pi}{\cosh^2(\pi r)} \, dr \\
        =& \frac{\pi}{4} \zeta_{K}(-1) \sum_{n=0}^{\infty} \frac{(-1)^n t^{n-1}}{n !} \int_{-\infty}^{\infty} \frac{(r^2+\frac{1}{4})^n}{\cosh^2(\pi r)} \, dr \\
        =& \frac{a_{-1}(2)}{t} + a_0(2) +o(1) \quad (t \to +0).
\end{align*}
We calculate the coefficients $a_n(2)$ $(n=-1,0)$.
\[ a_{-1}(2) = \frac{\pi}{4} \zeta_{K}(-1) \int_{-\infty}^{\infty} \frac{dr}{\cosh^2(\pi r)}
= \frac{\pi}{4} \zeta_{K}(-1) \cdot \frac{4}{\pi} \int_{0}^{\infty} \frac{x}{(x^2+1)^2} \, dx
= \frac{1}{2} \zeta_{K}(-1),
\]
\begin{align*}
a_{0}(2) =& - \frac{\pi}{4} \zeta_{K}(-1) \Bigl \{ \int_{-\infty}^{\infty} \frac{r^2}{\cosh^2(\pi r)} \, dr  
+\frac{1}{4}\int_{-\infty}^{\infty} \frac{dr}{\cosh^2(\pi r)}  
\Bigr\}
= - \frac{\pi}{4} \zeta_{K}(-1) \Bigl( \frac{1}{6 \pi} + \frac{1}{4} \cdot \frac{2}{\pi} \Bigr) \\
=& -\frac{1}{6} \zeta_{K}(-1).
\end{align*}
Here, we used the formula: $\displaystyle{\int_{0}^{\infty} \frac{r^2}{\cosh^2(\pi r)} \, dr = \frac{(2^2-2)\pi^2}{(2 \pi)^2 \pi } \cdot \frac{1}{6}=\frac{1}{12 \pi}}$ on \cite[3.527 no.5]{GR}.
Besides, we calculate the coefficient $b_0(2)$ appearing in $E_2(t)$.   
\begin{align*}
b_0(2) &= - \sum_{R(\theta_1,\theta_2) \in 
\Gamma_{\mathrm{E}}} 
\frac{i e^{ - i \theta_1}}{8 \nu_{R} \sin \theta_1}
\int_{-\infty}^{\infty}   \frac{1}{\sqrt{4 \pi }} \exp \bigl( -\frac{u^2}{4} \bigr) \, 
\Bigl[  \frac{1 - e^{2 i \theta_1}}{1 - \cos 2 \theta_1} \Bigr] du  \\
&= - \sum_{j=1}^{N} \sum_{k=1}^{\nu_j-1} \frac{1}{4 \nu_j} \cdot \frac{1}{1-\cos \bigl( \frac{2 \pi k}{\nu_j}\bigr)} 
= - \sum_{j=1}^{N} \frac{\nu_j^2-1}{24 \nu_j}.
\end{align*}
Summing up each terms appearing in the right hand side of (\ref{theta2-tr}), we have the desired formula (\ref{theta2}).

Let us prove (\ref{thetam}) with $m=4$. For $t>0$,  
we also take the pair of test functions $h_1(r)=e^{-t(r^2+1/4)}$ and 
$g_1(u) = \frac{1}{\sqrt{4 \pi t}} \exp \bigl( -\frac{t}{4} -\frac{u^2}{4t} \bigr)$
in Proposition \ref{ddtrfm} with $m=4$,  
then we have
\begin{equation} \label{theta4-tr}
\theta_4(t)-\theta_2(t) + 1
= I_4(t)+E_4(t)+ HE_4(t)+HS_4(t). 
\end{equation}
Here, 
\begin{itemize}
\item $I_4(t)=\frac{\mathrm{vol}(\Gamma_K \backslash \bH^2)}{8 \pi^2}
\int_{-\infty}^{\infty}
   \exp \bigl( -t(r^2+1/4) \bigr)  \, r \tanh (\pi r)  \, dr$,
\item $E_4(t)= - \sum_{R(\theta_1,\theta_2) \in 
\Gamma_{\mathrm{E}}} 
\frac{i e^{ - i \theta_1} e^{2 i \theta_2}}{4 \nu_{R} \sin \theta_1}
\int_{-\infty}^{\infty}   \frac{1}{\sqrt{4 \pi t}} \exp \bigl( -\frac{t}{4} -\frac{u^2}{4t} \bigr) \, 
e^{-u/2} \bigl[  \frac{e^u - e^{2 i \theta_1}}{\cosh u - \cos 2 \theta_1} \bigr] du$, 
\item $HE_4(t) = - \sum_{(\gamma,\omega) \in \Gamma_{\mathrm{HE}}} 
\frac{\log N(\gamma_0) }{N(\gamma)^{1/2}-N(\gamma)^{-1/2}} \,  \frac{1}{\sqrt{4 \pi t}} \exp \bigl( -\frac{t}{4} -\frac{(\log N(\gamma))^2}{4t} \bigr) e^{2 i \omega}$,
\item $HS_4(t)=- 2 \log \varepsilon 
\sum_{k=1}^{\infty}  \frac{1}{\sqrt{4 \pi t}} \exp \bigl( -\frac{t}{4} -\frac{(2k \log \varepsilon)^2}{4t} \bigr) 
\bigl( \varepsilon^{-3k}  - \varepsilon^{-k} \bigr)$.
\end{itemize}
Similarly, we see that $HE_4(t)$ and $HS_4(t)$ are exponentially decreasing as $t \to +0$,    
and there is a constant $b_0(4)$ such that $E_4(t)=b_0(4)+o(1)$ $(t \to +0)$, and
$I_4(t) = \zeta_{K}(-1) \bigl( 1/t -1/3 \bigr)+o(1)$ $(t \to +0)$.
Summing up each terms appearing in the right hand side of (\ref{theta4-tr})
and using (\ref{theta2}) in the left side, we have the desired formula (\ref{thetam}) with $m=4$.
One can prove (\ref{thetam}) for $m \ge 6$ similarly. We complete the proof.
\end{proof}

\begin{proposition} \label{spectral-zeta}
Let $s$ be a fixed sufficiently large real number.
For $m \in 2 \N$, let 
\[ \zeta_m(w, s) := \sum_{n=0}^{\infty} \frac{1}{\bigl( \lambda_n(m)+s(s-1) \bigr)^w}  
\quad  (\RE(w) \gg 0)  .\]
be the spectral zeta function for $\square_m$.
Then $\zeta_m(w, s)$ is holomorphic at $w=0$.

\end{proposition}
\begin{proof}
We follow \cite[p.448]{E1}.
For $w \in \C$ with $\RE(w) \gg 0$, we have
\begin{equation}
\zeta_m(w, s) = \frac{1}{\Gamma(w)} \int_{0}^{\infty} \theta_m(t) e^{-s(s-1)t} t^{w} \, \frac{dt}{t}.
\end{equation}
We consider the first three terms of $\theta_m(t)$ in Proposition \ref{theta-asym}.
Let 
\begin{equation} \label{eta}
\eta_p(w,s) := \frac{1}{\Gamma(w)} \int_{0}^{\infty} t^{-p} e^{-s(s-1)t} t^{w-1} \, dt
 = \frac{1}{\Gamma(w)} \bigl( s(s-1) \bigr)^{p-w} \Gamma(w-p)
\end{equation}
with $p=0,\frac{1}{2},1$.  Then we see that $\eta_p(w,s)$ $(p=0,\frac{1}{2},1)$ are holomorphic at $w=0$. 
The reminder term is 
\begin{equation} \label{etaf}
\eta_{f}(w,s) : = \frac{1}{\Gamma(w)} \int_{0}^{\infty} f(t) e^{-s(s-1)t} t^{w} \, \frac{dt}{t}
\end{equation}
with $f(t)=o(1)$ $(t \to +0)$ and $O(1)$ $(t \to \infty)$.  
Since $\frac{1}{\Gamma(w)}$ vanishes at $w=0$, it completes the proof.
\end{proof}

\begin{proposition}  \label{det-asym}
Let $m$ be an even natural number. We have
\begin{equation} \label{det2-asym}
\begin{split}
- \frac{\partial}{\partial w} \zeta_{2}(w,s) \Big|_{w=0}
= &  - \zeta_{K}(-1)  \Bigl( s^2-s+\frac{1}{3} \Bigr) \log s 
+\frac{1}{2} \zeta_{K}(-1) \cdot s^2 -s \log \varepsilon \\
& +\Bigl( 2b_0(2) +2 \Bigr) \log s -\frac{1}{4} \zeta_{K}(-1) +\frac{1}{2} \log \varepsilon 
+o(1) \quad (s \to \infty),
\end{split}             
\end{equation}
and for $m \ge 4$, 
\begin{equation} \label{detm-asym}
\begin{split}
- \frac{\partial}{\partial w} \zeta_{m}(w,s) \Big|_{w=0}
=& - (m-1)\zeta_{K}(-1)  \Bigl( s^2-s+\frac{1}{3} \Bigr) \log s 
+\frac{m-1}{2} \zeta_{K}(-1) \cdot s^2 \\
& -s \log \varepsilon +\Bigl( 2b_0(2) + \cdots +2b_0(m) \Bigr) \log s \\ 
& -\frac{m-1}{4} \zeta_{K}(-1) +\frac{1}{2} \log \varepsilon +o(1) \quad (s \to \infty). 
\end{split}             
\end{equation}

Besides, we have for $m \ge 4$,
\begin{equation} \label{detmm-2-asym}
\begin{split}
- &\frac{\partial}{\partial w}  \zeta_{m}(w,s) \Big|_{w=0}  
  + \frac{\partial}{\partial w} \zeta_{m-2}(w,s) \Big|_{w=0}  \\
& =  -2 \zeta_{K}(-1)  \Bigl( s^2-s+\frac{1}{3} \Bigr) \log s 
+ \zeta_{K}(-1) \cdot s^2  
+\Bigl( 2b_0(m) -2 \, \delta_{4,m} \Bigr)  \log s \\
& \quad -\frac{1}{2} \zeta_{K}(-1)  
+o(1) \quad (s \to \infty). 
\end{split}             
\end{equation}
\end{proposition}
\begin{proof}
By the formulas (\ref{eta}) and (\ref{etaf}), we find that
\begin{align*} 
\frac{\partial}{\partial w} \eta_{0}(w,s) \Big|_{w=0} =& - \log \bigl( s(s-1) \bigr) = -2 \log s +o(1) \quad (s \to \infty), \\
 \frac{\partial}{\partial w} \eta_{\frac{1}{2}}(w,s) \Big|_{w=0} =& - 2 \sqrt{\pi} \bigl( s(s-1) \bigr)^{\frac{1}{2}}
=  - 2 \sqrt{\pi} \Bigl( s-\frac{1}{2} \Bigr)  +o(1) \quad (s \to \infty), \\
 \frac{\partial}{\partial w} \eta_{1}(w,s) \Big|_{w=0} =& s(s-1) \Bigl( \log \bigl( s(s-1) \bigr) -1 \Bigr) \\
 =& 2s (s-1) \log s +\frac{1}{2} -s^2 +o(1) \quad (s \to \infty), \\          
 \frac{\partial}{\partial w} \eta_{f}(w,s) \Big|_{w=0} =& o(1) \quad (s \to \infty).
\end{align*}

Therefore, by using (\ref{theta2}), we have
\begin{align*}
- \frac{\partial}{\partial w} \zeta_{2}(w,s) \Big|_{w=0}
=& -\frac{1}{2} \zeta_{K}(-1) \Bigl(2s (s-1) \log s +\frac{1}{2} -s^2\Bigr)  - \Bigl(s-\frac{1}{2} \Bigr) \log \varepsilon \\
& + \Bigl ( -\frac{1}{6} \zeta_{K}(-1) +b_0(2)+1 \Bigr) \cdot 2 \log s
+o(1) \quad (s \to \infty) \\
=& - \zeta_{K}(-1)  \Bigl( s^2-s+\frac{1}{3} \Bigr) \log s 
+\frac{1}{2} \zeta_{K}(-1) \cdot s^2 -s \log \varepsilon \\
& +\Bigl( 2b_0(2) +2 \Bigr) \log s -\frac{1}{4} \zeta_{K}(-1) +\frac{1}{2} \log \varepsilon 
+o(1) \quad (s \to \infty). 
\end{align*}
For $m \ge 4$, by using (\ref{thetam}), we have
\begin{align*}
- \frac{\partial}{\partial w} \zeta_{m}(w,s) \Big|_{w=0}
=& - (m-1)\zeta_{K}(-1)  \Bigl( s^2-s+\frac{1}{3} \Bigr) \log s 
+\frac{m-1}{2} \zeta_{K}(-1) \cdot s^2 -s \log \varepsilon \\
& +\Bigl( 2b_0(2) + \cdots +2b_0(m) \Bigr) \log s -\frac{m-1}{4} \zeta_{K}(-1) +\frac{1}{2} \log \varepsilon \\
& +o(1) \quad (s \to \infty).  
\end{align*}
We complete the proof.
\end{proof}

\section{Proof of Main Theorem} 

In this section we prove Theorem \ref{mainth}.
We prove the following two propositions.
The first proposition connect the completed Selberg zeta functions:
\[ \widehat{Z}_2^{\frac12}(s), \, \widehat{Z}_4(s), \dots, \widehat{Z}_m(s) \]
with the regularized determinants of Laplacians: 
\[ \mathrm{Det} \bigl( \square_2+s(s-1) \bigr), \,  \mathrm{Det} \bigl( \square_4+s(s-1) \bigr), \dots, \mathrm{Det} \bigl( \square_m+s(s-1) \bigr). \] 
The second proposition determines the explicit relations among them.
Theorem \ref{mainth} is deduced from these two propositions.

\begin{proposition}
Let $\square_m :=  \Delta_{0}^{(1)} \big|_{V_{m}^{(2)}}$ for $m \in 2 \N$. 
There exit polynomials $P_2(s),\ldots,P_m(s)$ such that
\[ \widehat{Z}_2^{\frac12}(s)=e^{P_2(s)}\, \frac{\mathrm{Det} \bigl( \square_2+s(s-1) \bigr)}{s(s-1)}, \quad 
\widehat{Z}_4(s) = e^{P_4(s)}  \,
\frac{ s(s-1) \cdot \mathrm{Det} \bigl( \square_4+s(s-1) \bigr)}{\mathrm{Det} \bigl( \square_2+s(s-1) \bigr)}, \]  
\[ \widehat{Z}_m(s) = e^{P_m(s)}  
\frac{\mathrm{Det} \bigl( \square_m+s(s-1) \bigr)}{\mathrm{Det} \bigl( \square_{m-2}+s(s-1) \bigr)}
\quad (m \ge 6).
\]
\end{proposition}
\begin{proof}
Let $k$ be a sufficiently large natural number. We note that
\[  \Bigl( - \frac{1}{2s-1} \frac{d}{ds} \Bigr)^{k+1} \zeta_m(w,s) 
= w(w+1) \cdots (w+k) \, \zeta_m(w+k+1,s). \]  

Taking $\displaystyle{ -\frac{\partial}{\partial w}\Big|_{w=0}}$ of both sides, 
we have
\begin{equation} \label{det1}
\Bigl( - \frac{1}{2s-1} \frac{d}{ds} \Bigr)^{k+1} \log \mathrm{Det} \bigl( \square_m+s(s-1) \bigr) 
= - \sum_{j=0}^{\infty} \frac{k !}{ \bigl( \lambda_j(m)+s(s-1) \bigr)^{k+1}}.
\end{equation}
Let $m=2$, we use the following double differences of STF with the certain test function:
(See \cite[Theorem 6.4]{G1}.)
\begin{align*}
&  \sum_{j=0}^{\infty} \Bigl[ \frac{1}{\rho_j(2)^2+(s-\frac{1}{2})^2} 
 + \sum_{h=1}^{2}\frac{c_h(s)}{\rho_j(2)^2 + \beta_h^2} \Bigr]
-  \Bigl[ \frac{1}{s(s-1)} 
 + \sum_{h=1}^{2}\frac{c_h(s)}{\beta_h^2-\frac{1}{4}} \Bigr] \\
& =  \zeta_{K}(-1) \sum_{k=0}^{\infty}
\Bigl[ \frac{1}{s+k} 
+ \sum_{h=1}^{2} \frac{c_h(s)}{\beta_h+\frac{1}{2}+k} \Bigr]  
+ \frac{1}{2s-1}  \frac{\frac{d}{ds}\sqrt{Z_{2}(s)}}{\sqrt{Z_{2}(s)}}
+\sum_{h=1}^{2} \frac{c_h(s)}{2 \beta_h} 
\frac{\frac{d}{d \beta_h}\sqrt{Z_{2}(\frac{1}{2}+\beta_h)}}{\sqrt{Z_{2}(\frac{1}{2}+\beta_h)}}
\\
& \quad + \frac{1}{2s-1} 
\sum_{j=1}^{N} \sum_{l=0}^{\nu_j-1} \frac{\nu_j-1-2l}{2\nu_j^2}
\, \psi \Bigl( \frac{s+l}{\nu_j} \Bigr)
+ \sum_{h=1}^{2} \frac{c_h(s)}{2 \beta_h} 
\sum_{j=1}^{N} \sum_{l=0}^{\nu_j-1} \frac{\nu_j-1-2l}{2\nu_j^2}
\, \psi \Bigl( \frac{\frac{1}{2}+\beta_h+l}{\nu_j} \Bigr)
\\
& \quad + \frac{1}{2s-1} \frac{d}{ds} 
\log \bigl(  \varepsilon^{-s}   \bigr) 
+ \sum_{h=1}^{2} \frac{c_h(s)}{2 \beta_h} \frac{d}{d \beta_h} 
\log \bigl( 
\varepsilon^{-(\beta_h +1/2)}
\bigr) \\
& \quad + \frac{1}{2s-1} \frac{d}{ds} 
\log \biggl\{ \frac{1}
{(1-\varepsilon^{-2s})} \biggr\} 
+ \sum_{h=1}^{2} \frac{c_h(s)}{2 \beta_h} \frac{d}{d \beta_h} 
\log \biggl\{ 
\frac{1}
{(1-\varepsilon^{-(2 \beta_h +1)})}
\biggr\}.
\end{align*}
Here, $\psi(z)$ is the digamma function, $\beta_1 \ne \beta_2$ are constants and $c_1(s),c_2(s)$ are quadratic polynomials
invariant under $s \to 1-s$. 
Operating $\displaystyle{\Bigl( - \frac{1}{2s-1} \frac{d}{ds} \Bigr)^{k} }$ on both sides, 
we have
\begin{equation} \label{det2}
 \sum_{j=0}^{\infty} \frac{k !}{ \bigl( \lambda_j(2)+s(s-1) \bigr)^{k+1}}
=\Bigl( - \frac{1}{2s-1} \frac{d}{ds} \Bigr)^{k} 
\frac{1}{2s-1} \frac{d}{ds} \log 
\Bigl( \widehat{Z}_2^{\frac12}(s) \cdot s(s-1) \Bigr).
\end{equation}
By (\ref{det1}) and (\ref{det2}), we have
\[ \Bigl( - \frac{1}{2s-1} \frac{d}{ds} \Bigr)^{k+1} \log \mathrm{Det} \bigl( \square_2+s(s-1) \bigr) 
= \Bigl( - \frac{1}{2s-1} \frac{d}{ds} \Bigr)^{k+1}  \log \Bigl( \widehat{Z}_2^{\frac12}(s) \cdot s(s-1) \Bigr). \]

Therefore, we find that there exists a polynomial $P_2(s)$ such that
\begin{equation} \label{p2}
\log \mathrm{Det} \bigl( \square_2+s(s-1) \bigr)  + P_2(s)
  = \log \Bigl( \widehat{Z}_2^{\frac12}(s) \cdot s(s-1) \Bigr).
\end{equation}
Thus we have
\[ \widehat{Z}_2^{\frac12}(s)=e^{P_2(s)}\, \frac{\mathrm{Det} \bigl( \square_2+s(s-1) \bigr)}{s(s-1)}. \]

Let $m \ge 4$ be an even integer. We use the following double differences of STF with the certain test function:
(See \cite[Theorem 5.2]{G1}.)
\begin{align*}
& \sum_{j=0}^{\infty} \Bigl[ \frac{1}{\rho_j(m)^2+(s-\frac{1}{2})^2} 
 + \sum_{h=1}^{2}\frac{c_h(s)}{\rho_j(m)^2 + \beta_h^2} \Bigr] \\
& \quad - \sum_{j=0}^{\infty} \Bigl[ \frac{1}{\rho_j(m-2)^2+(s-\frac{1}{2})^2} 
 + \sum_{h=1}^{2} \frac{c_h(s)}{\rho_j({m-2})^2 + \beta_h^2} \Bigr] 
+ \delta_{m,4} \Bigl[ \frac{1}{s(s-1)} 
 + \sum_{h=1}^{2}\frac{c_h(s)}{\beta_h^2-\frac{1}{4}} \Bigr] \\
& = 2 \zeta_{K}(-1) \sum_{k=0}^{\infty}
\Bigl[ \frac{1}{s+k} 
+ \sum_{h=1}^{2} \frac{c_h(s)}{\beta_h+\frac{1}{2}+k} \Bigr] 
+ \frac{1}{2s-1} \frac{Z_{m}'(s)}{Z_{m}(s)}
+\sum_{h=1}^{2} \frac{c_h(s)}{2 \beta_h} 
\frac{Z_{m}'(\frac{1}{2}+\beta_h)}{Z_{m}(\frac{1}{2}+\beta_h)}
\\
& \quad + \frac{1}{2s-1} 
\sum_{j=1}^{N} \sum_{l=0}^{\nu_j-1} \frac{\nu_j-1-\alpha_l(m,j)-\overline{\alpha_l}(m,j)}{\nu_j^2}
\, \psi \Bigl( \frac{s+l}{\nu_j} \Bigr) \\
& \quad 
+ \sum_{h=1}^{2} \frac{c_h(s)}{2 \beta_h} 
\sum_{j=1}^{N} \sum_{l=0}^{\nu_j-1} \frac{\nu_j-1-\alpha_l(m,j)-\overline{\alpha_l}(m,j)}{\nu_j^2}
\, \psi \Bigl( \frac{\frac{1}{2}+\beta_h+l}{\nu_j} \Bigr) 
\\
& \quad + \frac{1}{2s-1} \frac{d}{ds} 
\log \biggl\{ \frac{(1-\varepsilon^{-(2s+m-4)})}
{(1-\varepsilon^{-(2s+m-2)})} \biggr\} 
+ \sum_{h=1}^{2} \frac{c_h(s)}{2 \beta_h} \frac{d}{d \beta_h} 
\log \biggl\{ 
\frac{(1-\varepsilon^{-(2 \beta _h +m-3)})}
{(1-\varepsilon^{-(2 \beta_h +m-1)})}
\biggr\}. 
\end{align*}
Here, $\beta_1 \ne \beta_2$ are constants and $c_1(s),c_2(s)$ are quadratic polynomials
invariant under $s \to 1-s$. 

Operating $\displaystyle{\Bigl( - \frac{1}{2s-1} \frac{d}{ds} \Bigr)^{k} }$ on both sides, 
we have
\begin{equation} \label{det4}
\begin{split}
 &\sum_{j=0}^{\infty} \frac{k !}{ \bigl( \lambda_j(m)+s(s-1) \bigr)^{k+1}}
  -\sum_{j=0}^{\infty} \frac{k !}{ \bigl( \lambda_j(m-2)+s(s-1) \bigr)^{k+1}} \\
&=\Bigl( - \frac{1}{2s-1} \frac{d}{ds} \Bigr)^{k} 
\frac{1}{2s-1} \frac{d}{ds} 
\Bigl( \log \widehat{Z}_m(s)  - \delta_{m,4} \log \bigl( s(s-1) \bigr) \Bigr).
\end{split}
\end{equation}
By (\ref{det1}) and (\ref{det4}), there exists a polynomial $P_m(s)$ such that
\begin{equation} \label{pm}
\begin{split}
& \log \mathrm{Det} \bigl( \square_{m}+s(s-1) \bigr)  
   -\log \mathrm{Det} \bigl( \square_{m-2}+s(s-1) \bigr)  
+ P_m(s) \\
&  = \log \widehat{Z}_m(s) - \delta_{m,4} \log \bigl( s(s-1) \bigr) .
  \end{split}
\end{equation}
We complete the proof.
\end{proof}

\begin{proposition}
We have
\begin{align*}
P_2(s)  =& \Bigl( s-\frac{1}{2} \Bigr)^2\zeta_{K}(-1) - \frac{1}{2} \log \varepsilon + \sum_{j=1}^{N}\frac{\nu_j^2-1}{12 \nu_j}\log \nu_j, \\ 
P_m(s) =& 2 \Bigl( s-\frac{1}{2} \Bigr)^2\zeta_{K}(-1) + \sum_{j=1}^{N} \frac{\nu_j^2-1 -12\alpha_0(m,j) \bigl\{ \nu_j - \alpha_0(m,j)  \bigr\} }{6 \nu_j}\log \nu_j \quad  (m \ge 4).
\end{align*}
\end{proposition}
\begin{proof}
Substituting (\ref{z2-asym}) and (\ref{det2-asym}) in (\ref{p2}), we have
\begin{align*}
P_2(s) =&  \log \Bigl( \widehat{Z}_2^{\frac12}(s) \cdot s(s-1) \Bigr)
           - \log \mathrm{Det} \bigl( \square_2+s(s-1) \bigr) \\
          =&   \zeta_{K}(-1) \biggl\{ \frac{3}{2} s^2 
          - s -\Bigl( s^2-s+\frac{1}{3} \Bigr) \log s \biggr\} 
          - \sum_{j=1}^{N} \frac{\nu_j^2-1}{12 \nu_j} \log \frac{s}{\nu_j}  -s \log \varepsilon \\
          &+2 \log s + \zeta_{K}(-1)  \Bigl( s^2-s+\frac{1}{3} \Bigr) \log s 
           -\frac{1}{2} \zeta_{K}(-1) \cdot s^2 +s \log \varepsilon \\
          & -\Bigl( 2b_0(2) +2 \Bigr) \log s +\frac{1}{4} \zeta_{K}(-1) -\frac{1}{2} \log \varepsilon 
          +o(1) \quad (s \to \infty) \\
          =& \Bigl( s-\frac{1}{2} \Bigr)^2\zeta_{K}(-1) - \frac{1}{2} \log \varepsilon + \sum_{j=1}^{N}\frac{\nu_j^2-1}{12 \nu_j}\log \nu_j 
            +o(1) \quad (s \to \infty).
\end{align*}
Since $P_2(s)$ is a polynomial, we have the desired formula for $P_2(s)$. 

Let $m \ge 4$. 
Substituting (\ref{zm-asym}) and (\ref{detmm-2-asym}) in (\ref{pm}), we have
\begin{align*}
P_m(s) =&  \log \widehat{Z}_m(s)   - \delta_{m,4}  \log \Bigl( s(s-1) \Bigr)  \\
             &-\log \mathrm{Det} \bigl( \square_{m}+s(s-1) \bigr) 
               +\log \mathrm{Det} \bigl( \square_{m-2}+s(s-1) \bigr) \\
          =& 2 \Bigl( s-\frac{1}{2} \Bigr)^2\zeta_{K}(-1) + \sum_{j=1}^{N} \frac{\nu_j^2-1 -12\alpha_0(m,j) \bigl\{ \nu_j - \alpha_0(m,j)  \bigr\} }{6 \nu_j}\log \nu_j  
            +o(1) \quad (s \to \infty).
\end{align*}
Since $P_m(s)$ is a polynomial, we have the desired formula for $P_m(s)$. 
It completes the proof.
\end{proof}

\bibliographystyle{plain}
\def\cprime{$'$} \def\polhk#1{\setbox0=\hbox{#1}{\ooalign{\hidewidth
  \lower1.5ex\hbox{`}\hidewidth\crcr\unhbox0}}}
\providecommand{\bysame}{\leavevmode\hbox
to3em{\hrulefill}\thinspace}
\providecommand{\MR}{\relax\ifhmode\unskip\space\fi MR }
\providecommand{\MRhref}[2]{%
  \href{http://www.ams.org/mathscinet-getitem?mr=#1}{#2}
} \providecommand{\href}[2]{#2}

\end{document}